%% file: main.tex
\documentclass[a4paper,11pt,twoside]{article}
\usepackage{fullpage}
\usepackage{amsmath,amssymb,amsthm,amsfonts,amsopn}
\usepackage{xfrac}
\usepackage[shortlabels]{enumitem} 
\usepackage{mathrsfs}
\usepackage{theoremref} 
\usepackage[all]{xy} 
\usepackage{tikz-cd}

\usepackage{dutchcal}
\DeclareMathAlphabet{\dutchcal}{U}{dutchcal}{m}{n}

\newtheorem{thm}{Theorem}[section]
\newtheorem{conj}{Conjecture}[section]
\newtheorem{lem}[thm]{Lemma}
\newtheorem{thmdef}[thm]{Theorem-Definition}
\newtheorem{prop}[thm]{Proposition}
\newtheorem{cor}[thm]{Corollary}
\theoremstyle{definition}
\newtheorem{ex}[thm]{Example}
\newtheorem{defn}[thm]{Definition}

\newtheorem{rmk}[thm]{Remark}

\setlist[enumerate,1]{label = \upshape{(\alph *)}, ref = \text{\upshape{(\alph *)}}}

\setlist[enumerate,2]{label = \upshape{(\roman *)}, ref = \text{\upshape{(\roman *)}}}

\setlist[enumerate,3]{label = \upshape{(\arabic *)}, ref = \text{\upshape{(\arabic *)}}}

\newcommand{\adj}{\dashv}

\newcommand{\functor}[1]{\operatorname{#1}}
\newcommand{\op}{^{\functor{op}}}
\newcommand{\co}{^{\functor{co}}}

\newcommand{\Ker}{\functor{Ker}}

\newcommand{\Hom}{\functor{Hom}}

\newcommand{\fixedcat}[1]{\operatorname{#1}}
\newcommand{\Ab}{\fixedcat{Ab}}
\newcommand{\Set}{\fixedcat{Set}}

\newcommand{\Mod}{\fixedcat{Mod}}
\renewcommand{\mod}{\fixedcat{mod}}
\newcommand{\MODCAT}{\mathbb{MODCAT}}
\newcommand{\MODCATCO}{\mathbb{MODCAT}\co}

\newcommand{\Funl}[1]{\fixedcat{Fun}_{\lambda}{\left(#1\right)}}

\newcommand{\Fun}[1]{\fixedcat{Fun}{\left(#1\right)}}
\newcommand{\fun}[1]{\fixedcat{fun}{\left(#1\right)}}
\newcommand{\funl}[1]{\fixedcat{fun}_{\lambda}{\left(#1\right)}}

\newcommand{\funi}[1]{\fixedcat{fun_\infty}(#1)}
\newcommand{\Zg}[1]{\fixedcat{Zg}(#1)}
\newcommand{\Presl}[1]{\fixedcat{Pres}_\lambda (#1)}
\newcommand{\Presm}[1]{\fixedcat{Pres}_\mu (#1)}
\newcommand{\fp}[1]{\fixedcat{Pres}_\omega (#1)}
\newcommand{\ev}[1]{\fixedcat{ev}_{#1}}

\newcommand{\category}[1]{\mathcal{#1}}
\newcommand{\B}{\category{B}}
\newcommand{\C}{\category{C}}
\newcommand{\D}{\category{D}}
\newcommand{\E}{\category{E}}
\newcommand{\F}{\category{F}}
\newcommand{\G}{\category{G}}

\newcommand{\I}{\category{I}}
\newcommand{\J}{\category{J}}
\newcommand{\K}{\category{K}}

\newcommand{\M}{\category{M}}
\newcommand{\N}{\category{N}}

\newcommand{\R}{\category{R}}
\renewcommand{\S}{\category{S}}
\newcommand{\T}{\category{T}}
\newcommand{\U}{\category{U}}

\newcommand{\ZZ}{\mathbb{Z}}

\newcommand{\ka}{{k}}
\newcommand{\la}{{l}}
\newcommand{\ma}{{m}}
\newcommand{\na}{{n}}

\renewcommand{\phi}{\varphi}
\newcommand{\lamdba}{\lambda}

\renewcommand{\subset}{\subseteq}

\newcommand{\extdir}[1]{\overrightarrow{{#1}}}
\newcommand{\extinv}[1]{\overleftarrow{{#1}}}

\newcommand{\setc}[2]{\left\{ #1  \ : \  #2 \right\}}

\newcommand{\rest}[2]{\left. {#1}\right\vert_{#2}}

\begin{document}
    \title{Functors from the infinitary model theory of modules and the Auslander-Gruson-Jensen 2-functor} \author{Samuel Dean}
    \date{\today}
    \maketitle
    \begin{center}
        \emph{For Jeremy Russell.}
    \end{center}
        
    \abstract{We define the notion of a $\lambda$-definable category, a generalisation of the notion of definable category from the model theory of modules. Let $\C$ be a $\lambda$-accessible additive category. We characterise the additive functors $\C\to\Ab$ which preserve $\lambda$-directed colimits and products, by showing that they are the finitely presented functors determined by a morphism between $\lambda$-presented objects (the same result appears, for the case $\lambda=\omega$, in \cite{prest2011}, but we give a proof for any infinite regular cardinal $\lambda$). We remark that \cite{arb} shows that every $\lambda$-definable subcategory of $\C$ is the class of zeroes of some set of such functors, thus obtaining a $\lambda$-ary generalisation of the finitary ($\lambda = \omega$) result from the finitary model theory of modules.

    We show that, to analyse the $\lambda$-ary model theory of a locally $\lambda$-presentable additive category $\C$, it is sufficient to consider \emph{finitary} pp formulas in the language of right $\Presl\C$-modules, where $\Presl\C$ is the category of $\lambda$-presented objects of $\C$, with the caveat that these pp formulas are interpreted among right $\Presl\C$-modules which preserve $\lambda$-small products. In particular, for an additive category $\R$ with $\lambda$-small products (e.g. $\R=\Presl\C\op$ for $\C$ a $\lambda$-presented additive category), the $\lambda$-accessible functors $\N\to\Ab$ which preserve products are precisely the finitely accessible functors $\R\Mod\to\Ab$ which preserve products, restricted to $\N$, where $\N\subseteq\R\Mod$ is the category of left $\R$-modules which preserve $\lambda$-small products.
    }
    \tableofcontents

    \section*{Introduction}\label{sec:introduction}
    \input{sec-introduction}

    \section{Accessible functors and $\lambda$-definable subcategories}\label{sec:smallstuff}
    \input{sec-accessible-functors}

    \section{Functor categories from the model theory of modules}\label{sec:functor-categories-from-the-model-theory-of-modules}
    \input{sec-functor-categories}

    \section{The Auslander-Gruson-Jensen 2-functor}\label{sec:the-auslander-gruson-jensen-2-functor}
    \input{the-auslander-gruson-jensen-2-functor}

\bibliographystyle{acm}
\bibliography{slayer}
\end{document}

%% file: sec-introduction.tex
In short, the \emph{raison d'\^e tre} of this paper is to carry out the chore of setting up the infinitary model theory of modules, at least in terms of appropriate functor categories.
Applications, to abstract analysis, contramodules, and dualisation of accessible functors, will come in future work.
As we shall discuss, the generalisation is not always as obvious as simply putting one's favourite Greek letter in front of every word in the model theory of modules, but thankfully the literature exists to carry some results forward in this manner.

Let $\R$ be a small preadditive category.
A \emph{left $\R$-module} is an additive functor $\R\to\Ab$.
We will write $\R\Mod$ for the category of left $\R$-modules.
A \emph{right $\R$-module} is an additive functor $\R\op\to\Ab$, where $\R\op$ denotes the opposite of $\R$.
We will write $\Mod\R$ for teh category of right $\R$-modules.

Let $\D$ be a full subcategory of $\M \ :\ =\R\Mod$.
We will write $\fun\M$ for the category of functors $\M\to\Ab$ which preserve directed colimits and products\footnote{It will be made clear why $\fun\M$ is a legitimate category. In fact, $\fun\M$ is equivalent to a small category.}.
The following are equivalent (see \cite{prest2009}):
\begin{itemize}
    \item $\D$ is closed under directed colimits, products, and pure subobjects.
    \item $\D=\setc{M\in \M}{\forall F\in\F,FM=0}$ for some full subcategory $\F\subset\fun\M$.
\end{itemize}
If either of these equivalent conditions are satisfied, we say that $\D$ is a \emph{definable subcategory} of $\M$.
These comments work as well when $\M$ is replaced by an arbitrary finitely accessible additive category.
Note that there is just a set of definable subcategories of $\M$.
In fact, there is a canonical topology on this set, and the resulting space $\Zg\M$ is called the \emph{Ziegler spectrum of $\M$} (or \emph{of $\R$} when $\M=\R\Mod$).
The Cantor-Bendixson rank of $\Zg\M$ is one of a few closely related measures of complexity of $\M$; See~\cite[Chapter 27]{prest2011} for a discussion.

In Section~\ref{sec:smallstuff}, we shall explain the set-theoretical legitimacy of considering categories of functors with large codomain, so long as a size condition, such as accessibility, is placed on the functors.
This has already been studied in the enriched context in~\cite{daylack}, but in Section~\ref{subsec:smallfuncats}, we will also provide an additional characterisation of \emph{small functors} (small colimits of representables) in the $\Set$-enriched context, by showing that $F:\C\to\Set$ is small if and only if there is a cofinal functor $\J\to\E_F$ where $\J$ is small and $\E_F$ is the category of elements of $F$.
We apply this characterisation to show that small functors $\C\to\Set$ are closed under left Kan extensions along any functor $\B\to\C$, with no restriction on the size of $\B$.
We then give a ``weak representability'' criterion, a tool for us to cut down the presentation of a functor, which is really just a careful analysis of Freyd's representability theorem~\cite[p. 121]{maclane}.

Let $\lambda$ be a regular infinite cardinal.
In Section~\ref{subsec:fpstuff}, we characterise the $\lambda$-accessible product preserving functors $\C\to\Ab$, where $\C$ is additive and $\lambda$-accessible, and we point out that $\lambda$-definable subcategories of $\C$ can be obtained using those functors.
Here \textit{$\lambda$-definable subcategory }is defined as an obvious extension of the notion of \textit{definable subcategory}: a subcategory closed under $\lambda$-directed colimits, $\lambda$-pure subobjects, and products.

Let $\R$ be a pre-additive category with $\lambda$-small products.
We show that a module $M\in\R\Mod$ preserves $\lambda$-small products if and only if it is a $\lambda$-directed colimit $M=\int^{i\in I}M_i$ of finitely presented modules $M_i$ (which, to the author's current knowledge, is not presently in the literature).
If $\N\subseteq\R\Mod$ is the category of all such modules, Section~\ref{subsec:suff} shows that a functor $F:\N\to\Ab$ preserves $\lambda$-directed colimits and products if and only if it is the restriction to $\N$ of some pp-pair in the (finitary) language of left $\R$-modules.
Thus, so long as they are interpreted among modules which preserve $\lambda$-small products, this shows that pp-pairs are sufficient for $\lambda$-ary module theory of modules.
It not difficult to see that a $\lambda$-definable subcategory $\D\subseteq \N$ can be found as the class of modules in $\N$ which make some pp-pairs vanish, but we establish this stronger result about the appearance of the functors in $\funl\N$, the functors $\N\to\Ab$ which preserve products and $\lambda$-directed colimits.

%% file: sec-accessible-functors.tex
\subsection{Small functors}\label{subsec:smallfuncats}
\input{subsec-small-functors}

\subsection{Accessible functors and $\lambda$-definable subcategories}\label{subsec:accfuncats}
\input{subsec-accessible-functors-definable-subcats}

\subsection{A presentability theorem}\label{subsec:a-presentability-theorem}
\input{subsec-representability-theorem}

%% file: subsec-small-functors.tex
\begin{defn}
    Let $\J$, $\C$ and $\D$ be categories.
    \begin{enumerate}[(a)]
        \item For a functor $D:\J\to\C$ we will write
        \[\int_{J\in\J}DJ\text{ and }\int^{J\in\J}DJ\]
        for its limit and colimit, if they exist. \emph{(We do not require $\J$ to be small.)}
        \item For a functor $D:\J\times\C\to \D$, we write
        \[\int^{J\in \J}D(J,-) \ : \ \C\to \D \ : \ C\ \mapsto \int^{J\in \J}D(J,C)\]
        for the \emph{colimit functor} if the colimit of $D(-,C)$ exists for each $C\in\C$.
    \end{enumerate}
\end{defn}
\begin{defn}
    Let $\C$ be a category.
    A functor $F:\C\to\Set$ is \emph{small} if it is the colimit functor of a small diagram of representables.
    That is, $F$ is small if
    \begin{displaymath}
        F\simeq\int^{J\in\J\op}\Hom_\C\left(DJ,-\right)
    \end{displaymath}
    for a functor $D:\J\to\C$ where $\J$ is small.
\end{defn}
\begin{defn}
    The \emph{category of elements} of a functor $F:\C\to\Set$ has:
    \begin{itemize}
        \item \emph{Objects} $(A,x)$ given by an object $A\in\C$ and an element $x\in FA$.
        \item \emph{Morphisms} $f:(A,x)\to(B,y)$ given by morphisms $f:A\to B\in \C$ such that $y=(Ff)x$.
    \end{itemize}
    We write $\E_F$ for the category of elements of $F$.
\end{defn}
\begin{defn}
    We say that objects $A,B\in\C$ in a category $\C$ are \emph{connected} if between any two objects $A,B\in\C$ there is a path of arrows and inverse arrows
    \[A = A_0\leftarrow A_1\rightarrow A_2\leftarrow A_3\rightarrow \dots \leftarrow A_n= B\in\C.\]
    This induces notions of \emph{connected component} and \emph{connected category}.
\end{defn}
\begin{prop}[{\cite[III.7]{maclane}}\thlabel{cancol} for case when $\C$ is small]
    For any functor $F:\C\to\Set$,
    \begin{displaymath}
        F\simeq \int^{(A,x)\in\E\op_F} \Hom_\C(A,-).
    \end{displaymath}
\end{prop}
\begin{proof}
    One can show that, for each $B\in\C$, the maps
    \[\Hom_\C(A,B)\to FB\ :\ f\ \mapsto \ (Ff)x\qquad((A,x)\in\E_F\op)\]
    are colimiting.
    Therefore $F$ is the required colimit functor.
\end{proof}
\begin{rmk}
    \thref{cancol} also has an additive analog where $\C$ is an additive category, $F$ is an additive functor $\C\to\Ab$ and $\Ab$ replaces $\Set$.
    Note that we really need $\C$ to be additive for this: If $\C$ is preadditive and not additive then $F$ is the canonical colimit of \emph{finite direct sums} of representables.
\end{rmk}
\begin{cor}[Well-known]
    \thlabel{conncolimit}For any functor $F:\C\to\Set$, the colimit
    \begin{displaymath}
        \int^{C\in\C}FC
    \end{displaymath}
    exists if and only if $\E_F$ has a set of connected components, and if it does exist then the colimit is the set of path components of $\E_F$.
\end{cor}
\begin{defn}
    We say that a functor $H:\J\to\K$ is \emph{final} if, for any object $K\in\K$
    \begin{displaymath}
        \int^{J\in\J}\Hom_\K(K,HJ)\simeq 1.
    \end{displaymath}
    For a subcategory $\J\subseteq \K$ of $\K$, if the inclusion $\J\to \K$ is final, we say that $\J$ is a \emph{final subcategory} of $K$.
    The dual notion of final is \emph{cofinal}.
    (The words final and cofinal are sometimes used contrariwise.)
\end{defn}
\begin{prop}[{\cite[IX.3]{maclane}}]
    The following conditions on a functor $H:\J\to\K$ are equivalent:
    \begin{enumerate}
        \item $H$ is final.
        \item For any object $K\in\K$, the comma category $K\downarrow H$ is non-empty and connected.
        \item For any functor $G:\K\to \C$, any object $C\in\C$ and any cone $\varphi:GH\to C$, there is a unique cone $\psi:G\to C$ such that $\psi H=\varphi$.
        \item For any functor $G:\K\to\C$, the colimit
        \[\int^{K\in\K}GK\]
        exists if and only if the colimit
        \[\int^{J\in\J}GHJ\]
        exists, and these colimits are isomorphic when they do exist.
    \end{enumerate}
\end{prop}
\begin{thm}
    \thlabel{smallcofinal}
    A functor $F:\C\to\Set$ is small if and only if there is a small category $\J$ and a cofinal functor $\J\to\E_F$.
\end{thm}
\begin{proof}
    We know that $F$ can be written as the colimit of a diagram of representables, indexed by $\E\op_F$:
    \begin{displaymath}
        F=\int^{(A,x)\in \E\op_F}\Hom_\C(A,-).
    \end{displaymath}
    Therefore, if there is such a cofinal functor $E:\J\to\E_J$ then
    \begin{displaymath}
        F\simeq\int^{J\in\J\op}\Hom_\C(UEJ,-),
    \end{displaymath}
    where $U:\E_F\to\C$ is the forgetful functor.

    For the converse, if $F$ is small then
    \begin{displaymath}
        F\simeq\int^{J\in\J\op}\Hom_\C(DJ,-)
    \end{displaymath}
    for some diagram $D:\J\to\C$ where $\J$ is small.
    Let
    \begin{displaymath}
        \phi_J:\Hom_\C(DJ,-)\to F
    \end{displaymath}
    be the component of the colimiting cone at each $J\in\J\op$.
    This corresponds to an element $x_J\in FDJ$ which is natural in $J\in\J\op$, and this defines a functor
    \[E:\J\to\E_F\ :\ J\ \mapsto \ (DJ,x_J).\]
    We will show that $E$ is cofinal.
    Let $(A,x)\in\E_F$ be given.
    An object in $E\downarrow (A,x)$ consists of an object $J\in \J$ morphism $f:DJ\to A$ such that $(Ff)x_J=x$.
    Suppose another such object $(K\in\J,g
    :DK\to A)$ is given.
    This amounts to
    \begin{displaymath}
    (\phi_J)
        _A(f)=x=(\phi_K)_A(g).
    \end{displaymath}
    However, $(\phi_J)_A$ and $(\phi_K)_A$ are components of a colimiting map which express the fact that $FA=\int^{J\in\J\op}\Hom_\C(DJ,A)$.
    Let $\D$ be the category of elements of the functor
    \[\Hom_\C(D-,A):\J\op\to\Set.\]
    Since $(J,f)$ and $(K,g)$ are objects of $\D$ which have the same image in the colimit, they must be connected as objects in $\D$ by \thref{conncolimit}, so there is a sequence of maps
    \begin{displaymath}
    \xymatrix{(J,f)=(J_0,f_0)&(J_1,f_1)\ar[l]\ar[r]&(J_2,f_2)&(J_3,f_3)\ar[l]\ar[r]&\dots&(J_n,f_n)=(K,g)\ar[l]}\end{displaymath}
    in $\D$, and every object in this path also must also have image $x\in FA$ in the colimit (by naturality of $\phi$). Therefore, the above is also a path in $E\downarrow (A,x)$.
    We have that $E\downarrow (A,x)$ is connected, since $(J,f)$ and $(K,g)$ are arbitrary objects.
\end{proof}
\begin{thm}
    For a functor $F:\C\to\Set$, the following are equivalent.
    \begin{enumerate}
        \item\label{itm:smalla} $F$ is the left Kan extension of $F|_\B$ along the inclusion $\B\hookrightarrow\C$, for a small full subcategory $\B\subseteq \C$.
        \item\label{itm:smallb} $F$ is the left Kan extension of a functor $\B\to \Set$ along a functor $\B\to\C$, for a small category $\B$.
        \item\label{itm:smallc} $F$ is the left Kan extension of a small functor $\B\to \Set$ along some functor $\B\to\C$.
        \item\label{itm:smalld} $F$ is small.
    \end{enumerate}
\end{thm}
\begin{proof}
    It is clear that $\ref{itm:smalla}\implies\ref{itm:smallb}\implies\ref{itm:smallc}$.

    Given a functor $E:\B\to\Set$, one can construct its left Kan extension $F:\C\to\Set$ along $P:\B\to\C$ as a colimit functor
    \begin{displaymath}
        F=\int^{(A,x)\in\E_E\op}\Hom_\C(PA,-),
    \end{displaymath}
    if this colimit functor exists.
    But, if $E$ is small, by \thref{smallcofinal}, there is a final functor $\J\op\to\E_E\op$ for some small category $\J$, and so this colimit functor can certainly be formed and $F$ is small.
    Therefore $\ref{itm:smallc}\implies \ref{itm:smalld}$.

    The fact $\ref{itm:smalld}\implies\ref{itm:smalla}$ is proven in~\cite{daylack}: If
    \begin{displaymath}
        F\simeq \int^{J\in\J\op} \Hom_\C(DJ,-)
    \end{displaymath}
    then one can take $\B=\{DJ\mid J\in\J\}$.
\end{proof}
\begin{defn}
    For a category $\C$ we write $(\C,\Set)$ for the \emph{category of small functors}.
    The legitimacy of this category is established by formally giving it the following data:
    \begin{itemize}
        \item \emph{Objects} given by pairs $(\J,E)$ where $\J$ is a small category and $E:\J\to\C$ is a functor.
        \item \emph{Morphisms} $(\J,E)\to (\K,F)$ given by elements of \begin{displaymath}
                                                                           \int_{J\in\J}\int^{K\in\K\op}\Hom_\C(FK,EJ).
        \end{displaymath}
        \item For any object $(\J,E)$, and any $J\in\J\op$, the identity $1_{EJ}\in\Hom_\C(EJ,EJ)$ gives a canonical element $(J, 1_{EJ})$ of $\Hom_\C(E-,EJ):\J\to\Set$; the path components of these elements, for each $J\in\J$, give rise to the \emph{identity}
        \begin{displaymath}
            1_{(\J,E)}\in  \int_{J\in\J}\int^{K\in\J\op}\Hom_\C(EK,EJ).
        \end{displaymath}
        \item For morphisms
        \begin{displaymath}
            \xymatrix{(\I,D)\ar[r]^\alpha&(\J,E)\ar[r]^\beta&(\K,F)}
        \end{displaymath}
        and any $I\in\I$, $\alpha_I$ is the path component of some $(J,f:EJ\to DI)$ in the category of elements of $\Hom_\C(E-,DI)$.
        But, $\beta_J$ is the path component of some $(K,g:FK\to EJ)$ in the category of elements of $\Hom_\C(F-,EJ)$.
        We set $(\beta\alpha)_I$ to the path component of $(K,fg:FK\to DI)$ in the category of elements of $\Hom(F-,DI)$.
        It is not difficult to show that this defines an element
        \begin{displaymath}
            \beta\alpha\in\int_{I\in\I}\int^{K\in\K\op}\Hom_\C(FK,DI),
        \end{displaymath}
        the \emph{composition} of $\beta$ and $\alpha$.
    \end{itemize}
    It is clear that we get a category $(\C,\Set)$ since morphisms $(\J,E)\to(\K,F)$ correspond to natural transformations
    \begin{displaymath}
        \int^{J\in\J\op}\Hom_\C(EJ,-)\to \int^{K\in\K\op}\Hom_\C(FK,-).
    \end{displaymath}
\end{defn}

%% file: subsec-accessible-functors-definable-subcats.tex
We defer to~\cite{ar} for our terminology.
Throughout, $\mu$ and $\lambda$ will refer to regular infinite cardinals. \emph{(If $\alpha$ is a singular cardinal, then a poset is $\alpha$-directed if and only it is $\alpha^+$-directed. Since all successor cardinals are regular, we may assume our indices of accessibility are so.)} When $\C$ is accessible, we will write $\Presl\C$ for its full subcategory of $\lambda$-presented objects.

\begin{defn}
    A functor $F:\C\to\D$ is \emph{$\lambda$-accessible} if $\C$ and $\D$ are $\lambda$-accessible and $F$ preserves $\lambda$-directed colimits; it is \emph{accessible} if it is $\lambda$-accessible for some $\lambda$.
\end{defn}

\begin{prop}
    \thlabel{lambdapresacc}
    Suppose $\D$ is a $\lambda$-accessible category.
    For a functor $F:\C\to\D$, the following are equivalent:
    \begin{enumerate}
        \item\label{itm:accb} $\Hom(D,F-):\C\to\Set$ is accessible for each $D\in\Presl\D$.
        \item\label{itm:accc} $F$ is accessible.
    \end{enumerate}
\end{prop}
\begin{proof}
    The proof of~\cite[Corollary 2.22]{ar} suffices.
    We have simply inferred a sharper result.
\end{proof}
\begin{prop}[{\cite[Examples 2.17 (2)]{ar}}]
    \thlabel{accsmall}A functor $F:\C\to\Set$ is accessible if and only if it is small.
\end{prop}
\thref{lambdapresacc} and \thref{accsmall} justify the following definition.
\begin{defn}
    For accessible categories $\C$ and $\D$, we write $(\C,\D)$ to denote the \emph{category of accessible functors} $\C\to\D$.

    Remark: Supposing that $\D$ is $\lambda$-accessible, the legitimacy of this category is established by giving it the formal definition as the full subcategory subcategory $(\C,\D)\subseteq(\Presl\D\op,(\C,\Set))$ consisting of functors $\Phi:\Presl\D\op\to(\C,\Set)$ such that the composition $\ev C\ \cdot\ \Phi:\Presl\D\op\to\Set$ is representable for each evaluation functor $\ev C:(\C,\Set)\to\Set$ ($C\in\C$). We should not much care that a particular index of accessibility $\lambda$ for $\D$ is used, as the resulting category $(\C,\D)$ is invariant under such choices.
\end{defn}
\begin{defn}
    Let $\C$ be an accessible category.
    A morphism $f:A\to B\in \C$ is said to be \emph{$\lambda$-pure} (or a \emph{$\lambda$-pure monomorphism}) if for any $g:C\to D\in\Presl\C$ and any $u:C\to A\in\C$,
    \begin{displaymath}
        fu\text{ factors through }g\ \Longrightarrow \ u\text{ factors through }g.
    \end{displaymath}
    We say that a full subcategory $\D\subseteq \C$ is \emph{closed under $\lambda$-pure subobjects} if, for any $\lambda$-pure $f:A\to B\in\C$, if $B\in\D$ then $A\in\D$.

    When $\lambda=\omega$, we say \emph{pure}, \emph{pure monomorphism} or \emph{closed under pure subobjects}, respectfully.

\end{defn}
\begin{defn}
    A \emph{$\lambda$-definable subcategory} $\D\subseteq\C$ of a locally $\lambda$-presentable category $\C$ is a full subcategory which is closed under products, $\lambda$-directed colimits, and $\lambda$-pure subobjects.
    A \emph{definable subcategory} of a finitely accessible category is an $\omega$-definable subcategory.
\end{defn}
\begin{rmk}
    Definable subcategories of module categories were originally introduced in the model theory of modules.
    In~\cite{kr} they are considered outside of the additive context.
\end{rmk}
\begin{thm}[{\cite[Corollary 2.32]{ar}} with an easy modification]
    \thlabel{mesharpen}
    Every $\lambda$-accessible category $\C$ is equivalent to a subcategory of $(\J,\Set)$ closed under $\lambda$-directed colimits and $\lambda$-pure subobjects, where $\J=(\Presl \C)\op$.
\end{thm}
\begin{proof}
    The proof is much the same as that of {\cite[Corollary 2.32]{ar}}, with some stronger observations.
    There is a fully faithful functor
    \begin{displaymath}
        E:\C\to(\J,\Set)
    \end{displaymath}
    given by $EC=\Hom_\C(-,C)|_\J$ for any $C\in\C$.
    If $C\in\Presl\C$ then $EC=\Hom_{\J}(C,-)$ is finitely presented.
    Let $\D$ be the image of $\C$ under this embedding.
    Since $\D$ is clearly closed under $\lambda$-directed colimits, we need only prove it is closed under $\lambda$-pure subobjects.

    Suppose $f:A\to B$ is a pure-monomorphism where $B\in\D$.
    The object $A$ is a canonical colimit $(u_i:A_i\to A)_{i\in I}$ of finitely presented (hence $\lambda$-presented) objects $A_i\in(\J,\Set)$.
    Since $\D$ is $\lambda$-accessible, for each $i\in I$ there is a $\lambda$-presented $B_i\in\D$ (and hence $B_i$ is a finitely presented in $(\J,\Set)$) and a commutative square
    \begin{displaymath}
        \xymatrix{A_i\ar[d]_{u_i}\ar[r]^{f_i}&B_i\ar[d]^{v_i}\\
        A\ar[r]_f&B}
    \end{displaymath}
    in $(\J,\Set)$.
    Since $f$ is $\lambda$-pure (it is $\omega$-pure), there is a morphism $\overline u_i:B_i\to A$ such that $\overline u_i f_i=u_i$, for each $i\in I$.
    It follows that the $\overline u_i$ form a cofinal subdiagram of the canonical colimit of $A$, and is therefore $\lambda$-filtered.
    Therefore $A$ is a $\lambda$-filtered colimit of objects $B_i\in \D$, so $A\in\D$.
\end{proof}
\begin{thm}[{\cite[Corollary 2.36]{ar}}]
    \thlabel{hard} Let $\C$ be an accessible category, and let $\D$ be an accessibly embedded subcategory of $\C$.
    Then $\D$ is accessible if and only if it is closed under $\lambda$-pure subobjects for some regular cardinal $\lambda$.
\end{thm}
\begin{defn}
    Let $\M$ be a class of morphisms in an accessible category $\C$.
    We say that $A\in\C$ is \emph{$\M$-injective} if every morphism $A\to B$ in $\M$ induces a surjective function $\Hom_\C(B,M)\to\Hom_\C(A,M)$.
    A \emph{small-injectivity class} in $\C$ is a full subcategory whose objects are the $\M$-injective objects for some set of maps $\M$; it is a \emph{$\lambda$-injectivity class} if the morphisms in $\M$ are between $\lambda$-presentable objects.
\end{defn}
\begin{thm}[{\cite[Theorem 4.8]{ar}}]
    \thlabel{wkref} Let $\D$ be a full subcategory of a locally presentable category $\C$.
    Then the following conditions are equivalent.
    \begin{enumerate}
        \item\label{itm:defa} $\D$ is a small-injectivity class in $\C$.
        \item\label{itm:defb} $\D$ is accessible, accessibly embedded, and closed under products in $\C$.
        \item $\D$ is weakly reflective and accessibly embedded in $\C$.
    \end{enumerate}
\end{thm}
\begin{thm}[{\cite[Characterisation Theorem 4.11]{ar}}]
    \thlabel{char} The following conditions on a category $\C$ are equivalent:
    \begin{enumerate}
        \item $\C$ is accessible and weakly cocomplete.
        \item $\C$ is accessible and has products.
        \item $\C$ is equivalent to a weakly reflective, accessibly embedded subcategory of $(\J,\Set)$ for some some small category $\J$.
        \item $\C$ is equivalent to a small-injectivity class in a locally presentable category.
    \end{enumerate}
\end{thm}
\begin{defn}
    \thlabel{definabledef}
    We say that an additive category $\D$ is \emph{$\lambda$-definable} if $\D\simeq\D'$ for some $\lambda$-definable subcategory $\D'$ of $\R\Mod$ for some preadditive category $\R$ (or, equivalently, a $\lambda$-definable cubcategory $\D'$ of some locally $\lambda$-presentable additive category).
\end{defn}
Many results in the theory of accessible categories carry with them the question of whether they can be sharpened.
For example, the equivalence $\ref{itm:defa}\Leftrightarrow\ref{itm:defb}$ in \thref{wkref} has since been sharpened to \thref{inj} below.
\begin{thm}[{\cite[Theorem 2.2]{arb}}]
    \thlabel{inj}
    Let $\C$ be a locally $\lambda$-presentable category.
    A full subcategory $\D\subseteq \C$ is a $\lambda$-injectivity class in $\C$ if and only if it is it is a $\lambda$-definable subcategory of $\C$.
\end{thm}
Counterexamples exist against various other such sharpening results.
Although all categories which are $\lambda$-definable for some $\lambda$ are accessible categories (see \thref{ldef} for $\lambda$-definable), this result cannot be sharpened.
For example, the category $\D$ of divisible abelian groups is a definable subcategory of $\Ab$, but, by the result~\cite[Korollar 1.8]{breit1970} of Breitsprecher and the classification of finitely generated abelian groups, $\fp\D=\fp\Ab  \cap \D=0$, so $\D$ is not finitely accessible.
So, although every $\lambda$-accessible category with products may be found as a $\lambda$-definable subcategory of a presheaf category, it is not the case that every $\lambda$-definable category is $\lambda$-accessible.
However, every $\omega$-definable subcategory of $\Ab$ is $\aleph_1$-accessible, since, by the Löwenheim-Skolem theorem, every abelian group is the ($\aleph_1$-directed) union of its countable (i.e. $\aleph_1$-presented) elementary subgroups.

%% file: subsec-representability-theorem.tex
In this section we present the following weakening of the general adjoint functor theorem.
We will use it to reduce the number of generators required by a functor.
\begin{thm}
    \thlabel{benny} Let $D:\J\to\C$ be such that the limit
    \[\int_{J\in\J}DJ\in\C\]
    exists and the colimit functor
    \[\int^{J\in \J\op}\Hom_\C\left(DJ, -\right):\C\to\Set\]
    exists.

    By the universal property of colimits, and the Yoneda lemma, there is a canonical morphism
    \begin{displaymath}
        \mathrm{can}_D:\int^{J\in \J\op}\Hom_\C\left(DJ, -\right)\to \Hom_\C\left(\int_{J\in\J}DJ,-\right).
    \end{displaymath}

    For a functor $F:\C\to\Set$, the following are equivalent:
    \begin{enumerate}
        \item\label{itm:preslim}$F$ preserves the limit $\int_{J\in\J}DJ$.
        \item\label{itm:uv} For any morphism
        \[u:\int^{J\in\J\op}\Hom_\C(DJ,-)\to F\]
        there is a unique morphism
        \[v:\Hom_\C\left(\int_{J\in\J}DJ,-\right)\to F\]
        such that $v\cdot\mathrm{can}_D=u$.

        In the situation $\ref{itm:uv}$, we also have:
        \begin{enumerate}
            \item\label{itm:uvi} If $u$ is an epimorphism then $v$ is an epimorphism.
            \item\label{itm:uvii} If $u$ is a split monomorphism then $v$ is a split monomorphism.
            \item\label{itm:uviii} If $u$ is an isomorphism then $v$ is an isomorphism.
        \end{enumerate}
    \end{enumerate}
\end{thm}
\begin{proof}
    The equivalence $\ref{itm:preslim}\Leftrightarrow\ref{itm:uv}$ is known and not hard to prove, but often only stated when $\C$ is small.
    The remarks on $\ref{itm:uv}$ are easy additions: $\ref{itm:uvi}$ is immediate; to get $\ref{itm:uvii}$, we apply $\ref{itm:uv}$ to $F=\Hom\left(\int_{J\in\J}DJ,-\right)$ and $u=\mathrm{can}_D$ to obtain the fact that, if $w$ is a left inverse of $u$, then $\mathrm{can}_D \cdot w$ is a left inverse of $v$; $\ref{itm:uviii}$ follows from $\ref{itm:uvi}$ and $\ref{itm:uvii}$.
\end{proof}
\begin{rmk}
    \thlabel{addbenny}
    An additive analog of \thref{benny} exists.
    We replace $\C$ by an additive category, $\Set$ by $\Ab$ and $F$ by an additive functor $\C\to\Ab$.
    We do not require $\J$ to be additive.
\end{rmk}
\begin{ex}
    We can deduce the representability theorems from \thref{benny}.
    Let a functor $F:\C\to\Set$ on a category $\C$ be given and let $\E=\E_F$.
    We know that $F$ can be written as the colimit of a diagram of representables, indexed by $\E$:
    \begin{displaymath}
        F=\int^{(A,x)\in \E\op}\Hom_\C(A,-).
    \end{displaymath}
    If $F$ has a solution set $\F\subseteq \E$, $\C$ is complete, and $F$ is continuous, then there is a small cofinal subcategory $\G\subseteq \E$ which contains $\F$, so
    \begin{displaymath}
        F\simeq \int^{(A,x)\in \G\op}\Hom_\C(A,-).
    \end{displaymath}
    It follows from \thref{benny}\ref{itm:uviii} that $F$ is representable.
\end{ex}

%% file: sec-functor-categories.tex
\subsection{Finitely presented functors}\label{subsec:fpstuff}
\input{subsec-fp-functors}

\subsection{pp-pairs}\label{subsec:ppp}
\input{subsec-pp-pairs}

\subsection{Sufficiency of finitary pp-pairs}\label{subsec:suff}
\input{subsec-suff-pp-pairs}


\subsection{The $\lambda$-pure exact structure}\label{subsec:the-lambda-pure-exact-structure}
\input{subsec-purity}

%% file: subsec-fp-functors.tex
In this section, let $\C$ denote an accessible additive category.
\begin{defn}
    A functor $F:\C\to\Ab$ is said to be \emph{finitely presented} if there is an exact sequence
    \begin{displaymath}
        \xymatrix{\Hom_\C(B,-)\ar[r]&\Hom_\C(A,-)\ar[r]&F\ar[r]&0}
    \end{displaymath}
    for objects $A,B\in\C$.
    We write $\fp{\C,\Ab}$ for the \emph{category of finitely presented functors} $\C\to\Ab$.
\end{defn}
\begin{ex}
    Suppose $\C$ has kernels and suppose $F:\C\to\Ab$ is finitely presented.
    It is known that, if $F$ is left exact $F$ is representable.
\end{ex}
We recall the following fundamental result of Auslander.
\begin{thm}[{\cite[Proposition 2.1]{auslander1965}}]
    \thlabel{aus} For any additive category $\C$, $\fp{\C,\Ab}$ has object-wise cokernels.
    If $\C$ has weak cokernels then $\fp{\C,\Ab}$ is abelian with object-wise kernels.
    Conversely, if $\fp{\C,\Ab}$ has weak kernels then $\C$ has weak cokernels.
    If $\C$ is abelian then $\fp{\C,\Ab}$ has global dimension $2$ or $0$.
\end{thm}
\begin{thm}
    \thlabel{fpacc}Let $\C$ be a $\lambda$-accessible category with products.
    For a functor $F:\C\to\Ab$, the following are equivalent.
    \begin{enumerate}
        \item\label{itm:fpa}$F$ is $\lambda$-accessible and preserves products.
        \item\label{itm:fpb}There is an exact sequence
        \begin{displaymath}
            \xymatrix{\Hom_\C(B,-)\ar[r]&\Hom_\C(A,-)\ar[r]&F\ar[r]&0.}
        \end{displaymath}
        where $A$ and $B$ are $\lambda$-presented.
    \end{enumerate}
\end{thm}
\begin{proof}
    $\Ab$ is $\lambda$-accessible.

    $\ref{itm:fpb}\implies\ref{itm:fpa}$: Products and $\lambda$-directed colimits are exact in $\Ab$.
    It follows that $F$ also preserves products $\lambda$-directed colimits.

    $\ref{itm:fpa}\implies\ref{itm:fpb}$: Let $\B=\Presl\C$ and consider the restriction $G=F|_\B:\B\to\Ab$.
    Then $F$ is the unique $\lambda$-accessible functor such that $F|_\B\simeq G$, and it can be recovered from $G$ as a colimit functor
    \begin{displaymath}
        F\simeq \int^{(A,x)\in\E_G\op}\Hom_\C(A,-).
    \end{displaymath}
    Therefore there is a set $\{A_i\}_{i\in I}$ of $\lambda$-presented objects of $\C$ and an epimorphism
    \begin{displaymath}
        \bigoplus_{i\in I}\Hom_\C(A_i,-)\to F.
    \end{displaymath}
    Now, $\C$ has the product $A=\prod_{i\in I}A_i$, and $F$ preserves it, so, by \thref{addbenny}, there is an epimorphism
    \begin{displaymath}
        u:\Hom_\C(A,-)\to F.
    \end{displaymath} Let $x=v_A(1_A)\in FA$ be the Yoneda-corresponding element to $v$.
    Since $v$ is an epimorphism there is no proper subfunctor $G\subset F$ such that $x\in GA$.
    In other words, $x$ generates $F$.

    Suppose $\{x_i\}_{i\in I}\in \bigoplus_{i\in I}\Hom_\C(A_i,A)$ are such that $u_A(\{x_i\}_{i\in I})=x$.
    Let $I'=\{i\in I:x_i\neq 0\}$, which is a finite set.
    We can restrict $u$ to an epimorphism \begin{displaymath}
                                                                                                                                                                                                            u':\Hom_\C(A',-)\to F.
    \end{displaymath}
    where $A'=\bigoplus_{i\in I'}A_i\in\B$.
    Now, $K=\Ker(u')$ preserves $\lambda$-directed colimits and products, so by applying the same argument to $K$ we obtain an epimorphism $\Hom_\C(B,-)\to K$ for some $B\in\B$.
\end{proof}
\begin{defn}
    For an accessible category $\C$ with products, we write $\funi\C$ for the category of accessible functors $\C\to\Ab$ that preserve products.
\end{defn}
\begin{defn}
    \thlabel{ldef}
    If $\D$ is a $\lambda$-definable additive category, we write $\Funl\D$ for the category of accessible functors $\D\to\Ab$ that preserve $\lambda$-directed colimits, and we write $\funl\D$ for the category of functors $\D\to\Ab$ that preserve $\lambda$-directed colimits and products.
    We simply write $\Fun\D$ and $\fun\D$ when $\lambda=\omega$.
\end{defn}
\begin{cor}
    For an accessible category $\C$ with products, $\fp{\C,\Ab}=\funi\C$.
\end{cor}
We can now restate \thref{inj} in the additive context.
\begin{cor}
    \thlabel{deffp}
    Let $\C$ be a locally $\lambda$-presentable additive category.
    A full subcategory $\D\subseteq \C$ is $\lambda$-definable if and only if there is a full subcategory $\F\subseteq\funl\C$ such that
    \begin{displaymath}
        \D = \setc{C\in\C }{ \forall F\in\F, F C = 0}.
    \end{displaymath}
\end{cor}
\begin{proof}
    Combine \thref{inj} and \thref{fpacc}
\end{proof}
\begin{rmk}
    Let $R$ be a ring and right $\M=R\Mod$ be the category of left $R$-modules.
    In light of Burke's result, which says that pp-pairs for left $R$-modules are precisely the finitely presented functors $\fp\M\to\Ab$ \cite{burke}, Prest essentially proved \thref{deffp}, in the case where $\lambda=\omega$ and $\C=\M$, at~\cite[Corollary 2.32]{prest1988}.
    Prest's proof utilises pp-elimination of quantifiers and Shelah's result which connects elementary equivalence with ultrapowers.
    Naively, one might try to generalise this proof to one which uses a generalised result about $\lambda$-complete ultrapowers.
    However, one meets some problems with this approach: The combinatorics for Shelah's result are very difficult, and the existence of non-principal $\lambda$-complete ultrapowers are not provable in ZFC~\cite[12.12]{jech}.
    So, one might be forgiven for thinking that the theory of accessible categories has saved us from dealing with these issues.
    Although it is certainly true that we have side-stepped the need for Shelah's result, we have not avoided the subtle combinatorics: These have simply been relocated, e.g.\ to the proof of \thref{hard}, on which \thref{deffp} relies.
\end{rmk}
\begin{rmk}\thlabel{rmk:funquotient}
    Suppose $\C$ is a locally presentable additive category and let $\D$ be a small-injectivity class in $\C$.
There is the obvious restriction functor
\begin{displaymath}
    i^*:\funi\C\to \funi\D
\end{displaymath}
and there is a unique right exact functor
\begin{displaymath}
    i_*:\funi\D\to \funi\C
\end{displaymath}
such that $i_*\Hom_\D(D,-)\simeq \Hom_\C(D,-)$, naturally in $D$.
It is easy to show that $i_*\adj i^*$ and $i_*$ is fully faithful.
We therefore obtain the following Serre quotient:
\begin{displaymath}
    \funi\C/\setc{F\in\funi\C}{F|_\D=0}\simeq \funi\D
\end{displaymath}
\end{rmk}
\begin{rmk}
It is unknown to me whether a sharpening of the result in \ref{rmk:funquotient}, 
\begin{displaymath}
    \funl\C/\setc{F\in\funl\C}{F|_\D=0}\simeq \funl\D
\end{displaymath}
is provable for any regular cardinal $\lambda$. The $\lambda=\omega$ is known \cite{prest2009}.
\end{rmk}

%% file: subsec-pp-pairs.tex
We now demonstrate the connection with finitary model theory.
We give an algebraic description of the syntactic aspects, but a model theorist would have no trouble translating.
No new results are given in this section, and~\cite{prest2009} or~\cite[Chapter 12]{prest1988} is an alternative reference.

In this subsection, let $\R$ denote a small additive category and let $\M=\R\Mod$.
We shall use lower case letters such as a $\ka$, $\la$, $\ma$ and $\na$ to denote objects of $\R$, and upper case letters such as $A$, $B$, $C$, and $D$, for arrows in $\R$.
For a left $\R$-module $M\in\M$ and an object $\ka\in\R$, we shall write $M^\ka=M\ka$, and for a morphism $A:\ma\to\na\in\R$ and an element $x\in M^\ma$, we shall write $Ax=(MA)x$.

\begin{defn}
    A \emph{pp-$\na$-formula (for left $\R$-modules)} is a cospan
    \begin{displaymath}
        \xymatrix{\phi:&\na \ar[r]^{A}& \ma & \la\ar[l]_{B}.}
    \end{displaymath}
    Given $M\in\M$, then $F_\phi M\subseteq M^\na$ is the set of elements $x\in M^\na$ such that $Ax=By$ for some $y\in M^\la$; this is the \emph{solution set of $\phi$ in $M$}.
    When $\la=0$ we say that $\phi$ is \emph{quantifier free}.
\end{defn}
\begin{thmdef}
    Every pp-$\na$-formula $\phi$ admits a finitely presented $\R$-module $M$ and an element $a\in F_\phi M$ such that, for any other module $N$ and solution $b\in F_\phi N$, there is a morphism $f:M\to N$ with $f^\na a=b$. \emph{Such a pair $(M,a)$ is called a \emph{free realisation of $\phi$}.}
\end{thmdef}
\begin{proof}
    Suppose $\phi$ is given by the span
    \begin{displaymath}
        \xymatrix{\phi:&\na\ar[r]^{A} & \ma & \la \ar[l]_{B}.}
    \end{displaymath}
    Let $M$ be defined by a pushout
    \begin{displaymath}
        \xymatrix{\Hom_\R(\ma,-)\ar[d]_{A^*}\ar[r]^{B^*}&\Hom_\R(\la,-)\ar[d]^c\\
        \Hom_\R(\na,-)\ar[r]_a&M.}
    \end{displaymath}
    Now if $N$ is a module with $b\in N^\na$ such that $Ab=Bd$ for some $d\in N^\la$, then there is a unique map $f:M\to N$ such that $f^\na a=b$ and $f^\la c=d$.
\end{proof}
If $\phi$ is a pp-$\na$-formula and $(M,a)$ is a free realisation, then the associated functor $F_\phi$ is the image of the induced morphism
\begin{displaymath}
    \Hom_\M(M,-)\to (-)^\na:g\mapsto ga,
\end{displaymath}
and therefore $F_\phi\in\fun{\M}$, because $\fun\M$ is an abelian subcategory of $(\M,\Ab)$ and so it is closed under images.

Let $M\in\fp\M$.
There is an exact sequence
\begin{displaymath}
    \xymatrix{\Hom_\R(\ma,-)\ar[r]^{A^*}&\Hom_\R(\na,-)\ar[r]&M\ar[r]&0}
\end{displaymath}
for some $A:\na\to\ma$, so $\Hom_\M(M,-)\simeq F_\phi$ where $\phi$ is
\begin{displaymath}
    \xymatrix{\phi:&\na\ar[r]^{A} & \ma & 0\ar[l].}
\end{displaymath}

Taking a subfunctor $G\subset F_\phi$ with $G\in\fun\M$, $G$ is clearly the image of some morphism ${\Hom_\M(N,-)\to(-)^\na}$, corresponding to an element $b\in N^\na$, with $N\in\fp\M$.
Suppose
\begin{displaymath}
    \xymatrix{
        \Hom_\R(\ka,-)\ar[r]^{B^*}&\Hom_\R(\la,-)\ar[r]& N\ar[r]&0
    }
\end{displaymath}
is an exact sequence for some $B:\la\to\ka\in\R$ and let $c\in N^\la$ be the element corresponding to the map ${\Hom_\R(\la,-)\to N}$.
There is ${C:\la\to\na\in\R}$ such that $b=Cc$.
It follows that $G=F_\psi$ where $\psi$ is the pp-$\na$-formula

\begin{displaymath}
    \xymatrix{\psi:&\na\ar[rr]^{\begin{pmatrix}
                                    1_\na\\0
    \end{pmatrix}} && \na\oplus \ka && \la\ar[ll]_{\begin{pmatrix}
                                                       C\\B
    \end{pmatrix}}.}
\end{displaymath}

\begin{defn}
    For two pp-$\na$-formulas $\phi$ and $\psi$, if $F_\psi$ is a subfunctor of $F_\phi$, we say that $\phi/\psi$ is a \emph{pp-$\na$-pair}.
    We also define $F_{\phi/\psi}=F_\phi/F_\psi\in\fun{\M}$.
\end{defn}

Now we have shown:

\begin{thm}[{\cite[Corollary 3.1.9]{burke}}]
    For any $F\in\fun\M$, there is a pp-pair $\phi/\psi$ such that $F\simeq F_{\phi/\psi}$ where $\phi$ is quantifier-free.
\end{thm}
\begin{rmk}
    Let $R$ be a ring.
    Traditionally, for a natural number $n$ and a ring $R$, pp-$n$-formulas for left $R$-modules are conditions of the form $\exists y, Ax=By$ where $x=(x_1,\dots,x_n)^T$ is a column of $n$ free variables and $y=(y_1,\dots,y_m)^T$ is a column of existentially quantified variables, and where $A$ and $B$ are appropriately sized matrices over $R$.
    Such a condition can be interpreted in a module $M$ to give a subgroup of $M^n$.
    The present definition of pp-pair recovers this one in the case that $\R$ is the category whose objects are natural numbers and an arrow {$m\to n$} is an $n\times m$ matrix, since there is an obvious equivalence $\R\Mod\simeq R\Mod$, under which pp-$n$-formulas for left $\R$-modules coincide with the usual notion of a pp-$n$-formula for left $R$-modules.
\end{rmk}

%% file: subsec-suff-pp-pairs.tex
{In this section, let $\R$ be a small additive category with $\lambda$-small products, let $\M=\R\Mod$, the category of left $\R$-modules, and let $\N$ be the full subcategory of $\M$ consisting of those left $\R$-modules which preserve $\lambda$-small products.}

\begin{rmk}
    Let $\C$ be a locally $\lambda$-presented addive category and let $\R=(\Presl\C)\op$: By~\cite[Proposition 2.2]{tendas}, $ \C$ consists of the left $\R$-modules which preserve $\lambda$-small limits.
    But then, as a subcategory of $\N$,
    \begin{displaymath}
        \C=\setc{N\in\N}{ N\text{ is left exact}}.
    \end{displaymath}
    Therefore $\C$ can be identified \emph{as the full subcategory of $\N$} consisting of those modules which make the functors given by pp-pairs
    \begin{displaymath}
        \frac{
            \xymatrix{A\ar[r]^{f}&B&0\ar[l]}
        }{
            \xymatrix{A\ar[r]^{1_{A}}&A &0\ar[l]}
        },
        \
        \frac{
            \xymatrix{B\ar[r]^{g}&C&0\ar[l]}
        }{
            \xymatrix{B\ar[r]^{1_{B}}&B&A\ar[l]_f}
        }
        \ (f\text{ the kernel of } g, g\in \R)
    \end{displaymath}
    vanish.
    For any $\lambda$-definable subcategory $\D\subseteq\C$, there is a set of functors $F_i:\C\to\Ab$ ($i\in I$), each of which is part of an exact sequence
    \begin{displaymath}
        \xymatrix{\Hom_\C(B_i,-)\ar[r]^{f_i^*}&\Hom_\C(A_i,-)\ar[r]&F_i\ar[r]&0}
    \end{displaymath}
    with $f_i:B_i\to A_i\in\R$, such that
    \begin{displaymath}
        \D=\{D\in\C \ : \ \forall i \in I, \ F_i D=0\}.
    \end{displaymath}
    Therefore, $\D$ can be found \emph{as a subcategory of} $\N$ (into which $\C$ embeds as a $\lambda$-definable reflective subcategory) as the modules which make both the above set of pp-pairs, and the pp-pairs
    \begin{displaymath}
        \frac{
            \xymatrix{A_i\ar[r]^{1_{A_i}}&A_i&A_i\ar[l]_{1_{A_i}}}
        }{
            \xymatrix{A_i\ar[r]^{1_{A_i}}&A_i&B_i\ar[l]_{f_i}}
        } \ (i\in I)
    \end{displaymath}
    vanish.
    This shows that any $\lambda$-definable additive category $\D$ (see \thref{definabledef}) can be found as the class of modules which make the functors, given by certain pp-pairs, vanish, so long as we consider $\D$ to be a $\lambda$-definable subcategory of $\N$.
    We will now strengthen this result to a description of the functors in $\funl\N$: We will show that they are precisely those functors $\M\to\Ab$, given by pp-pairs, for left $\R$-modules, restricted to $\N$.
\end{rmk}
\begin{thm}
    \thlabel{lpfp}
    $\Presl\N=\fp\M.$
\end{thm}
\begin{proof}
    Since $\N$ is the class of left $\R$-modules which are orthogonal to the set of all maps of the form
    \begin{displaymath}
        \bigoplus_{i\in I}\Hom_\R(\na_i,-)\to \Hom_\R\left(\prod_{i\in I}\na_i,-\right)
    \end{displaymath}
    with $|I|<\lambda$, it is a reflective $\lambda$-definable subcategory of $\M$ by~\cite[Theorem 1.39]{ar}, and in particular it is closed under $\lambda$-directed colimits.
    Therefore, $\N$ is a locally $\lambda$-presented category.
    Also, $\N$ is an abelian subcategory of $\M$.

    For any $M\in\Presl\M$, there is an exact sequence
    \begin{displaymath}
        \bigoplus_{i\in I}\Hom_\R(\ma_i,-)\to\bigoplus_{j\in J}\Hom_\R(\na_j,-)\to M\to 0,
    \end{displaymath}
    where $|I|,|J|<\lambda$.
    For any $N\in\N$, there are isomorphisms
    \begin{align*}
        \Hom_\M\left(\bigoplus_{i\in I}\Hom_\R(\ma_i,-),N\right)
        &\simeq\prod_{i\in I}\Hom_\M(\Hom_\R(\ma_i,-),N)\\
        &\simeq \prod_{i\in I}N\ma_i\\
        &\simeq N\left(\prod_{i\in I}\ma_i\right)\\
        &\simeq \Hom_\M\left(\Hom_\R\left(\prod_{i\in I}\ma_i,-\right),N\right)
    \end{align*}
    which are natural in $N$.
    This shows that, under the reflector $\M\to \N$,
    \begin{align*}
        \bigoplus_{i\in I}\Hom_\R(\ma_i,-)&\mapsto \Hom_\R\left(\prod_{i\in I}\ma_i,-\right)\\
        \bigoplus _{j\in J}\Hom_\R\left(\na_j,-\right)&\mapsto \Hom_\R\left(\prod_{j\in J}\na_j,-\right).
    \end{align*}
    Right exactness of the reflector now shows that the reflection of $M$ in $\N$ is a finitely presented left $\R$-module.
    But, since $\N$ is a reflective subcategory of $\M$ which is closed under $\lambda$-directed colimits, the $\lambda$-presented objects of $\N$ are precisely the class of summands of $\N$-reflections of $\lambda$-presented objects of $\M$.
    The class summands of finitely presented modules is closed under summands, so $\Presl\N\subseteq\fp{\M}$.
    The converse, that $\fp{\M}\subseteq\Presl\N$, is clear.
\end{proof}

\begin{cor}
    Let $\R$ be a small additive category with $\lambda$-small coproducts.
    For any left $\R$-module $M\in\R\Mod$, the following are equivalent.
    \begin{enumerate}
        \item\label{itm:fpthma} $M$ preserves $\lambda$-small products.
        \item\label{itm:fpthmb} $\fp\M/M$ is $\lambda$-filtered.
        \item\label{itm:fpthmc} $M$ is a $\lambda$-directed colimit of finitely presented modules.
    \end{enumerate}
\end{cor}
\begin{proof}
    Since $\Presl\N=\fp\M$, $\fp\M$ is closed, as a full subcategory of $\N$, under $\lambda$-small colimits.
    It easily follows that if $M\in\N$ then $\fp\M/M$ is $\lambda$-filtered.
    Therefore, $\ref{itm:fpthma}\implies\ref{itm:fpthmb}$.

    Since $M$ is the colimit of the forgetful functor $\fp\M/M\to \M$, if $\fp\M/M$ is $\lambda$-filtered then $M$ is a $\lambda$-filtered colimit, and hence a $\lambda$-directed colimit, of finitely presented modules.
    Therefore $\ref{itm:fpthmb}\implies\ref{itm:fpthmc}$.

    Since $\fp\M=\Presl\N\subset\N$ and $\N$ is closed under $\lambda$-directed colimits we have $\ref{itm:fpthmc}\implies\ref{itm:fpthma}$.
\end{proof}

\begin{cor}
    $\funl\N=\fun\M|_\N \ : \ = \setc{F|_\N}{F\in\fun\M}.$ Equivalently, for each $F\in\funl\N$ there is a pp-pair for left $\R$-modules such that $\phi/\psi$ such that $F\simeq F_{\phi/\psi}|_\N$.
\end{cor}

%% file: subsec-purity.tex
\begin{defn}
    Let $\C$ be a locally $\lambda$-presentable additive category. A sequence in $\C$,
    \begin{displaymath}
        \xymatrix{0\ar[r]&L\ar[r]^f&M\ar[r]^g&N\ar[r]&0,}
    \end{displaymath}
    is said to be \emph{$\lambda$-pure exact} if
    \begin{displaymath}
        \xymatrix{0\ar[r]&FL\ar[r]^{Ff}&FM\ar[r]^{Fg}&FN\ar[r]&0}
    \end{displaymath}
    is exact for any $F\in\funl\C$. When $\lambda=\omega$, we simply say \emph{pure exact}.
\end{defn}

\begin{defn}
    A map $g$ in a $\lambda$-accessible category $\C$ is called a \emph{$\lambda$-pure epimorphism} if $\Hom_\C(A,g)$ is surjective for all $A\in\Presl\C$.
\end{defn}

\begin{thm}
    \thlabel{pes}
    Let $\C$ be a locally $\lambda$-presented additive category. For a sequence
    \begin{displaymath}
        \xymatrix{\xi: & 0\ar[r]&L\ar[r]^f&M\ar[r]^g&N\ar[r]&0}
    \end{displaymath}
    in $\C$, the following are equivalent:
    \begin{enumerate}
        \item\label{itm:hompure} $\Hom_\C(A,\xi)$ is exact for any $A\in\Presl\C$.
        \item\label{itm:seqpure} $\xi$ is $\lambda$-pure.
        \item \label{itm:xipure} $\xi$ is a $\lambda$-directed colimit of split exact sequences (with first term $L$).
        \item\label{itm:fpure} $f$ is a $\lambda$-pure monomorphism and the sequence is a kernel-cokernel pair.
        \item \label{itm:gpure} $g$ is a $\lambda$-pure epimorphism and the sequence is a kernel-cokernel pair.
    \end{enumerate}
    The class of pure-exact sequences form an exact structure on $\C$.
    \begin{proof}
        Omitted -- the same moves can be made as for the case $\lambda=\omega$.
    \end{proof}
\end{thm}

\begin{ex}
    Let $R$ be a ring. For a family $A_i\in R\Mod$ $(i\in I)$,
    \begin{displaymath}
        \prod^\lambda_{i\in I}A_i \ : \ =\setc{(x_i)_{i\in I}\in \prod_{i\in I}}{\left|\setc{i\in I}{x_i\neq 0}\right|<\lambda}\subseteq \prod_{i\in I}A_i
    \end{displaymath}
    is a $\lambda$-pure submodule of the product. 
    
    When $R=\mathbb Z$, each $A_i=\mathbb Z$, $|I|$ is infinite and non-measurable, and $\lambda=\aleph_0$, it is known that this inclusion is not split: In that case, we have Zeeman's result \cite[Theorem 2(ii), p. 196, in the case $A=\bigoplus_I\ZZ$, $B=C=\ZZ$]{zeeman} which tells us $\Hom(\ZZ^I,Z)\simeq \bigoplus_I\ZZ$; if the inclusion $\bigoplus_I\ZZ\to\ZZ^I$ were split, we would have $\ZZ^I$ as a summand of $\bigoplus_I\ZZ$, while $|\ZZ^I|=2^{|I|}>|I|=\left|\bigoplus_I\ZZ\right|$.
\end{ex}

\begin{prop}
    Let $\R$ be an additive category with $\lambda$-small products and let $\N$ denote the category of left $\R$-modules which preserve $\lambda$-small products. 
    
    For a submodule $N\subseteq M$ with $N,M\in\N$, if $N$ is a pure submodule of $M$ (i.e. a pure subobject $N\subset M$ in $\R\Mod$) then it is a $\lambda$-pure submodule of $M$.
\end{prop}
\begin{proof}
 The sequence
        \begin{displaymath}
            \xymatrix{0\ar[r]&N\ar[r]&M\ar[r]&M/N\ar[r]&0}
        \end{displaymath}
        is pure in $\M=\R\Mod$. However, $N,M,M/N\in\N$ as $\N$ is an abelian subcategory of $\N$. Therefore, the exactness of this sequence is preserved by any functor in $\fun\M|_\N=\funl\N$. Therefore, the sequence is $\lambda$-pure.
\end{proof} 

\begin{defn}
    We say that an object in a locally $\lambda$-presented additive category $\C$ is \emph{$\lambda$-pure-injective} (respectively, \emph{$\lambda$-pure-projective}) if it is injective (respectively, projective) with respect to the $\lambda$-pure exact structure on $\C$. When $\lambda=\omega$, we simply say \textit{pure-injective} (respectively \textit{pure-projective}).
\end{defn}

\begin{thm}
    Let $\C$ be a locally $\lambda$-presented category and let $\R=(\Presl\C)\op$. For any object $C\in\C$,
    \begin{displaymath}
        C\text{ is }\lamdba\text{-pure-injective in }\C \iff \Hom_\C(-,C)|_\R\text{ is }\lamdba\text{-pure-injective in }\R\Mod.
    \end{displaymath}
\end{thm}
\begin{proof}
    Consider $\C$ to be the reflective full subcategory of $\R\Mod$ consisting of the modules which preserve $\lambda$-small limits. Let $R:\R\Mod\to\C$ be the left adjoint of the inclusion $\C\subseteq\R\Mod$.

    If $C$ is $\lambda$-pure-injective in $\R\Mod$, it is clearly $\lambda$-pure-injective in $\C$.

    Suppose $C$ is $\lambda$-pure-injective in $\C$. If
    \begin{displaymath}
        \xymatrix{0\ar[r]&L\ar[r]^f&M\ar[r]^g&N\ar[r]&0}
    \end{displaymath}
    is a $\lambda$-pure-exact sequence in $\R\Mod$ then
    \begin{displaymath}
        \xymatrix{0\ar[r]&RL\ar[r]^{Rf}&RM\ar[r]^{Rg}&RN\ar[r]&0}
    \end{displaymath}
    is $\lambda$-pure-exact in $\C$ since $R$ preserves $\lambda$-directed colimits (and therefore sequences that are $\lambda$-directed colimits of split sequences). There is an isomorphism of sequences
    \begin{displaymath}
        \xymatrix{0\ar[r]&\Hom_\C(RN,C)\ar[d]_{\simeq}\ar[rr]^{\Hom_\C(Rg,C)}&&\Hom_\C(RM,C)\ar[d]^{\simeq}\ar[rr]^{\Hom_\C(Rf,C)}&&\Hom_\C(RL,C)\ar[d]^{\simeq}\ar[r]&0
        \\
        0\ar[r]&\Hom_\C(N
        ,C)\ar[rr]^{\Hom_\C(g,C)}&&\Hom_\C(M,C)\ar[rr]^{\Hom_\C(f,C)}&&\Hom_\C(L,C)\ar[r]&0}
    \end{displaymath}
    with the top sequence exact, so the bottom sequence also is exact. Therefore $C$ is pure-injective in $\R\Mod$.

\end{proof}

%% file: the-auslander-gruson-jensen-2-functor.tex
\subsection{The colax 2-functor}\label{subsec:the-colax-2-functor}

\begin{defn}
    We write $\MODCAT$ for the following $2$-category:
    \begin{itemize}
        \item $0$-cells given by small pre-additive categories.
        \item $1$-cells $\R\to\S$ given by accessible additive functors $\R\Mod\to\S\Mod$.
        \item $2$-cells given by natural transformations.
    \end{itemize}
    We write $\MODCATCO$ for $\MODCAT$ with the $2$-cells reversed.
\end{defn}
\begin{defn}
    For any accessible functor $F:\R\Mod\to\Ab$, the \textbf{Auslander-Gruson-Jensen dual} of $F$ is the functor $D F:\Mod\R\to\Ab$ defined by
    \begin{displaymath}
    (D F)
        N=(F,N\otimes_\R -).
    \end{displaymath}
    for any $N\in\Mod\R$.
\end{defn}
\begin{rmk}
    Let $\R$ be a pre-additive category.
    
    Auslander, and independently Gruson and Jensen, established an equivalence
    \begin{displaymath}
    (\fun{\R\Mod})
        \op\to \fun{\Mod\R}:F\mapsto D_\R F
    \end{displaymath}
    given by
    \begin{displaymath}
    (DF)
        M=(F,M\otimes_\R -)
    \end{displaymath}
    with inverse also given by Auslander-Gruson-Jensen dual.
    This equivalence is the \textbf{Auslander-Gruson-Jensen duality}.

    In \cite{dean3} it was observed that this equivalence extends to a functor
    \begin{displaymath}
        \funi{\R\Mod}\op\to \Fun{\Mod\R}:F\mapsto D F
    \end{displaymath}
    with a fully faithful right and left adjoint (and hence, to a recollement), with the right adjoint given by Auslander-Gruson-Jensen dual, and the left adjoint being given by $L_0$ ($0$-th left derived functor) of the right adjoint.

    Now we further generalise the Auslander-Gruson-Jensen dual by showing that it defines a lax 2-functor
    \begin{displaymath}
        D:\MODCATCO\to\MODCAT.
    \end{displaymath}
\end{rmk}

\begin{prop}
    \thlabel{thm:dualaccessible}
    If $F:\R\Mod\to\Ab$ is $\lambda$-accessible then $D F:\Mod\R\to\Ab$ is $\mu$-accessible, where $\mu$ is any regular cardinal such that $F\in\Presm{\Funl{\R\Mod}}$.
    In particular, if $F$ is accessible then $D F$ is accessible.

\end{prop}
\begin{proof}
    Suppose $F$ preserves $\lambda$-directed colimits.
    Then $F\in \Funl{\R\Mod}$, which is a module category.
    Therefore $F$ is $\mu$-presented for some regular cardinal $\mu$.

    Since $D_\R F$ is the composition of the two functors
    \begin{displaymath}
        \xymatrix{\Mod\R\ar[rr]^{\hspace{-5mm}\text{tensor}}&&\Funl{\R\Mod}\ar[rr]^{\hspace{5mm}(F,-)}&&\Ab}
    \end{displaymath}
    which both preserve $\mu$-directed colimits.
\end{proof}

\begin{defn}
    Given any accessible functor $F:\R\Mod\to\S\Mod$, we define its \textbf{Auslander-Gruson-Jensen dual} to be the functor $DF:\Mod\R\to\Mod\S$ defined by
    \begin{displaymath}
        ((DF)N)S=D((F-)S)N
    \end{displaymath}
    for any $N\in\Mod\R$ and $S\in\S$, where the functor $(F-)S:\R\Mod\to\Ab$ is defined by $M\mapsto (FM)S$.
\end{defn}

\begin{cor}
    For any accessible additive functor $F:\R\Mod\to\S\Mod$, $DF:\Mod\R\to\Mod\S$ is accessible.
\end{cor}
\begin{proof}
    Suppose $F:\R\Mod\to\S\Mod$ is $\lambda$-presentable.
    For each $S\in\S$, if $(F-)S:\R\Mod\to\Ab$ is $\mu_S$-presentable in $\Funl{\R\Mod}$ then $D((F-)S)$ is $\mu_S$-accessible by \thref{thm:dualaccessible}.
    For any regular cardinal $\nu$ such that $\nu>\mu_S$ for all $S\in\S$, $DF$ is $\nu$-accessible.
\end{proof}

\begin{defn}
    \begin{itemize}
        \item For any $0$-cell $\R\in\MODCAT$ we write $D\R=\R\op$.
        \item For any $1$-cell $F:\R\to\S\in\MODCAT$ we write $DF:D\R\to D\S\in\MODCAT$ for Auslander-Gruson-Jensen dual of $F$.
        \item For any $2$-cell $\alpha:F\to G$ the induced $2$-cell $D\alpha:DG\to DF$.
    \end{itemize}
\end{defn}

\begin{prop}
    For $1$-cells $F:\R\to\S$, $G:\S\to\T$, there is a canonical natural transformation
    \begin{displaymath}
        \delta_{G,F}:(DG)(DF)\to D(FG)
    \end{displaymath}
    which is natural in $G$ and $F$.
\end{prop}

\begin{proof}
    Let $N\in\Mod\R,T\in\T$ be given.

    For a natural transformation $\alpha:\Phi\to\Psi$ between functors $\Phi,\Psi:\S\Mod\to\Ab$, we will write $\alpha*F$ for the induced natural transformation $\Phi F\to \Psi F$.

    We need to define a morphism

    \begin{align*}
        ((DG)(DF)N)T&=((DG)((DF)N))T
        \\& = D((G-)T)((DF)N)
        \\& = ((G-)T, (DF)N\otimes_\S -)
        \\&\to ((GF-)T,N\otimes_\R-)
        \\&=(D(GF)N)T.
    \end{align*}
    For any $\alpha:(G-)T\to (DF)N\otimes _\S-$, define
    \begin{displaymath}
        ((\delta_{G,F})_N)_T(\alpha)=\beta_{N,F}\circ (\alpha * F)
    \end{displaymath}
    where
    \begin{displaymath}
        \beta_{N,F}:(DF)N\otimes_\S(F-)\to N\otimes_\R-
    \end{displaymath}
    is given by, at each $S\in \S$, the component
    \begin{displaymath}
        \beta_{N,F,S}:((F-)S, N\otimes_\R -)\otimes_\ZZ (F-)S\to N\otimes_\R-
    \end{displaymath}
    such that, for each $M\in\R\Mod$,
    \begin{displaymath}
        \beta_{N,F,S,M}(\alpha\otimes_\ZZ x)=\alpha_M(x)\in N\otimes_\R M.
    \end{displaymath}
    Various naturality checks are necessary, but these are omitted.
\end{proof}

\begin{lem}
    \thlabel{ass}For $1$-cells $F:\R\to\S$, $G:\S\to\T$ and $H:\T\to\U$, the diagram
    \begin{displaymath}
        \xymatrix{(DH)(DG)(DF)\ar[d]_{\delta_{H,G}*DF}\ar[rr]^{DH*\delta_{G,F}}&&(DH)D(GF)\ar[d]^{\delta_{H,GF}}
            \\D(HG)(DF)\ar[rr]_{\delta_{HG,F}}&&D(HGF)}
    \end{displaymath}
    commutes.
\end{lem}

\begin{proof}
    Omitted; the proofs are not instructive, just a tedious diagram chase.
\end{proof}
\begin{defn}
    Writing $1_\R:\R\Mod\to \R\Mod$ and $1_{\R\op}:\Mod \R\to \Mod \R$ for the identity functors, we define a natural transformation
    \begin{displaymath}
        \psi_\R:1_{\R\op}\to D(1_\R)
    \end{displaymath}
    which, at any $N\in \Mod\R$, evaluates to the morphism
    \begin{displaymath}
        \psi_{\R,N}:N\to (D(1_\R))N
    \end{displaymath}
    defined by
    \begin{displaymath}
        (\psi_{\R,N,R}(x))_M(y)=x\otimes_\R y.
    \end{displaymath}
    for any $R\in\R$, $x\in NR$, $M\in \R\Mod$, and $y\in MR$.
\end{defn}

\begin{lem}
    \thlabel{unit1}
    For any $1$-cell $F:\R\to\S$, the diagrams
    \begin{displaymath}
        \xymatrix{DF\ar@{=}[rrd]\ar[rr]^{\psi_\S*DF}&&D(1_\S)(DF)\ar[d]^{\delta_{\S, F}}
            \\&&DF}
        \ \ \
        \xymatrix{DF\ar@{=}[rrd]\ar[rr]^{DF*\psi_\R}&&(DF)D(1_\R)\ar[d]^{\delta_{F,\R}}
            \\&&DS}
    \end{displaymath}
    commute.
\end{lem}
\begin{proof}
    Omitted; the proofs are not instructive, just a tedious diagram chase.
\end{proof}
\begin{thm}
    \thlabel{colax}
    The Auslander-Gruson-Jensen dual gives a lax $2$-functor
    \begin{displaymath}
    (D, \ \delta, \ \psi)
        :\MODCATCO\to\MODCAT.
    \end{displaymath}
\end{thm}
\begin{proof}
    \thref{ass} and \thref{unit1} state the appropriate axioms.
\end{proof}
Since lax $2$-functors preserve monads, we immediately gain the following result.
\begin{cor}
    If $L=(L,\ \Delta:L\to L^2,\ \epsilon:L\to 1_\R)$ is an accessible comonad on $\R\Mod$, then
    \begin{displaymath}
        DL \ : \ =(DL,\ (D\Delta)\circ \delta_{L,L},\ (D\epsilon)\circ \psi)
    \end{displaymath}
    is an accessible monad on $\Mod\R$.
\end{cor}

\subsection{Basic isomorphisms}\label{subsec:basic-isomorphisms}

\begin{rmk}
    We need to discuss various isomorphisms to do with tensor products. The lax 2-functor $D$ is defined on functors $\R\Mod\to\S\Mod$ where $\R$ and $\S$ are small pre-additive categories. To save notational anguish, in proofs and definitions we will assume that the pre-additive categories we discuss are actually rings. We trust that the reader will have no trouble generalising these if they so wish.

    We will write $\R,\S,\T,\U$ for arbitrary pre-additive categories as we require them.
\end{rmk}

\begin{defn}
    Let $k=\mathbb Q/\mathbb Z$.
    If $A$ is an abelian group we will define $A^*=\Hom_{\Ab}(A,k)$.
\end{defn}

\begin{defn}
 For any $A,M\in\M=\R\Mod$ there is a cananonical map
        \begin{align*}
            \tau_{M,A}:M^*\otimes_\R A&\to \Hom_\M(A,M)^*\\
            f\otimes_\R a&\mapsto (g\mapsto f(g(a)))
        \end{align*}
        which is natural in $A$ and $M$.
\end{defn}

\begin{rmk}\thlabel{niceomorphismlol}
    It is a basic well-known fact that,
    for any $A,M\in\R\Mod$, if $A$ is finitely presented then $\tau_{M,A}$ is an isomorphism.

    For any $M\in\R\Mod$, $N\in\Mod\R$, there is an isomorphism
    \begin{displaymath}
        (M,N^*)\simeq(N,M^*)
    \end{displaymath}
    which is natural in $M$ and $N$, by hom-tensor duality.
\end{rmk}
\begin{defn}Let $\R$ and $\S$ be small pre-additive categories.
\begin{itemize}
    \item     We will write $\extdir{(\R\Mod,\S\Mod)}$ for the category of functors $\R\Mod\to\S\Mod$ which preserve directed colimits.
    \item     We will write $\extinv{((\R\Mod)\op,\S\Mod)}$ for the category of functors $(\R\Mod)\op\to\S\Mod$ which preserve inverse colimits.
    \item     We will write $\R\mod=\fp{\R\Mod}$ and $\mod\R=\fp{\Mod\R}$.
\end{itemize} 
\end{defn}
\begin{rmk}\thlabel{rmk:equiv}
    For any $F\in(\R\mod,\S\Mod)$, we have the functor $\extdir{F}:\R\Mod\to\S\Mod$ defined by
    \begin{displaymath}
        \extdir F M=\Hom_{\R\Mod}(-,M)|_{(\R\mod)\op} \otimes_{\R\mod} F.
    \end{displaymath}
    By the ``tensor product" analog of the Yoneda lemma, there is an isomorphism $\rest{\extdir F}{\R\mod}\simeq F$.
    
     For any $H:\R\Mod\to\S\Mod$, there is a canonical natural transformation
     $$\alpha_H:\extdir{\rest{H}{\R\mod}}\to H$$
     such that $H$ preserves directed colimits if and only if $\alpha_H$ is an isomorphism.

     Therefore, restriction ot $\R\mod$ gives an equivalence
     \begin{displaymath}
         \extdir{(\R\Mod,\S\Mod)}\simeq (\R\mod,\S\Mod).
     \end{displaymath}
\end{rmk}

\begin{rmk}\thlabel{rmk:coequiv}
    For any $G\in((\R\mod)\op,\S\Mod)$, we have the functor $\extinv{G}:(\R\Mod)\op\to\S\Mod$ defined by
    \begin{displaymath}
        \extinv G M=(\Hom_{\R\Mod}(-,M)|_{(\R\mod)\op}, G).
    \end{displaymath}
    By the Yoneda lemma, there is an isomorphism $\rest{\extdir G}{(\R\mod)\op}\simeq G$.
    
     For any $K:(\R\Mod)\op\to\S\Mod$, there is a canonical natural transformation
     $$\beta_K:K\to \extinv{\rest{K}{(\R\mod)\op}}$$
     such that $K$ preserves inverse limits if and only if $\beta_K$ is an isomorphism.

     Therefore, restriction ot $\R\mod$ gives an equivalence
     \begin{displaymath}
         \extinv{((\R\Mod)\op,\S\Mod)}\simeq ((\R\mod)\op,\S\Mod).
     \end{displaymath}
\end{rmk}
\begin{cor}
    For any $F\in\extdir{(\R\Mod,\S\Mod)}$ and $M\in\R\Mod$, there is an isomorphism 
    \begin{displaymath}
        \chi_{F,M}:(DF)(M^*)\simeq (FM)^*
    \end{displaymath}
    natural in $F$ and $M$.
\end{cor}
\begin{proof}
    \begin{align*}
    (DF)(M^*)&=(F,M^*\otimes_\R-)
    \\
    & \simeq (\rest{F}{\R\mod}, \rest{M^*\otimes_\R-}{\R\mod})\text{ \thref{rmk:equiv}}
    \\
    & \simeq (\rest{F}{\R\mod}, \rest{\Hom(-,M)}{(\R\mod)\op}^*)\text{ \thref{niceomorphismlol}}
    \\
    & \simeq (\rest{\Hom(-,M)}{(\R\mod)\op}, \rest{F}{\R\mod}^*)
    \\
    & = \extinv{\rest{F}{\R\mod}^*}M
    \\ 
    & \simeq F^*M
    \\
    &=(FM)^*.
\end{align*}
\end{proof}

\begin{lem}
    Let $F:\R\Mod\to\S\Mod$ be accessible. The functor $DF$ is left pure exact, meaning that $DF$ sends pure exact sequences
    \begin{displaymath}
        \xymatrix{0\ar[r]&A\ar[r]& B\ar[r]&C\ar[r]&0}
    \end{displaymath}
    in $\Mod\R$ to a left exact sequence
    \begin{displaymath}
        \xymatrix{0\ar[r]&(DF)A\ar[r]&(DF) B\ar[r]&(DF)C}
    \end{displaymath}
    in $\Mod\S$.
    
    If $F$ preserves directed colimits then $DF$ preserves products. Regardless of  is left pure exact,
\end{lem}
\begin{proof}
    Left pure exactness is clear from the tensor characterisation of pure exact sequences: If $\xi$ is a pure exact sequence in $\Mod\R$, $\xi\otimes_\R-$ is exact, and so $(DF)\xi=(F,\xi\otimes_\R-)$ is left exact.
    
    For any $N\in\Mod\R$, $(DF)N=(F,N\otimes_\R-)\simeq (F|_{R\mod},N\otimes_\R-|_{\R\mod})$, naturally so. The tensor-by-finitely-presented-modules embedding $\Mod\R\to(\R\mod,\Ab)$ preserves products; this is clearly enough.
\end{proof}

\begin{cor}\thlabel{cor:wow}
    If $F:\R\Mod\to\S\Mod$ and $G:\S\Mod\to\T\Mod$ preserve directed colimits and $F$ preserves products, there is an isomorphism
    \begin{displaymath}
        (DG)(DF)\simeq D(GF).
    \end{displaymath}
\end{cor}
\begin{proof}
For any $M\in\R\Mod$, there are isomorphisms
    \begin{align*}
    (DG)(DF)(M^*) & \simeq (DG)((FM)^*)
    \\
    & \simeq (G(FM))^*
    \\
    & = (GFM)^*
    \\ & \simeq D(GF)(M^*)
\end{align*}
which are clearly natural in all variables.

The functors $DG$ and $D(GF)$ are left pure exact. If $F\in\fun{\R\Mod}$ then $DF\in\fun{\Mod\R}$; this is a statement about functors into $\Ab$ but clearly generalises into the statement that finitely accessible functors which preserve products dualise to finitely accessible functors which preserve products. Therefore $DF$ is pure \textit{exact}. It follows that $(DG)(DF)$ is left pure exact. Since $(DG)(DF)$ and $D(GF)$ agree on the $M^*$s, and there are enough pure injectives of the form $M^*$, this is enough.
\end{proof}

\begin{conj}
    The isomorphism in \thref{cor:wow} is $\delta_{G,F}$.
\end{conj}

\begin{ex}
    Let ${_\S}A{_\T}$ be an $\S$-$\T$-bimodule and let ${_\R}B{_\S}$ be an $\R$-$\S$-bimodule, finitely presented as a left $\R$-module. There is an isomorphism
    \begin{displaymath}
        D({_\T}A{_\S}\otimes{_\S}\Hom_{\R\Mod}({_\R}B{_\S},{_\R}-))\simeq \Hom_{\Mod\S}({_\T}A{_\S}, -\otimes_\R B{_\S}):\Mod\R\to \Mod \T.
    \end{displaymath}
\end{ex}
\begin{proof}
    Let $G={_\T} A_\S\otimes{_\S}-$, $F={_{\S}}\Hom_{\R\Mod}({_\R}B{_\S},{_\R}-)$ in \thref{cor:wow}.
\end{proof}